\pdfoutput=1
\RequirePackage{ifpdf}
\ifpdf 
\documentclass[pdftex]{sigma}
\else
\documentclass{sigma}
\fi

\newtheorem{thm}{Theorem}[section]
\newtheorem{pro}[thm]{Proposition}
\newtheorem{lem}[thm]{Lemma}

{\theoremstyle{definition}
\newtheorem{dfn}[thm]{Definition}
}

\numberwithin{equation}{section}

\newcommand\md[1]{\Delta_2^{(#1)}}
\newcommand\mo[1]{\Delta^{(#1)}}

\begin{document}

\allowdisplaybreaks

\newcommand{\arXivNumber}{1406.4645}

\renewcommand{\PaperNumber}{075}

\FirstPageHeading

\ShortArticleName{An Asymmetric Noncommutative Torus}

\ArticleName{An Asymmetric Noncommutative Torus}

\Author{Ludwik D\c{A}BROWSKI~$^\dag$ and Andrzej SITARZ~$^{\ddag\S}$}

\AuthorNameForHeading{L.~D\c{a}browski and A.~Sitarz}

\Address{$^\dag$~SISSA (Scuola Internazionale Superiore di Studi Avanzati),\\
\hphantom{$^\dag$}~via Bonomea 265, 34136 Trieste, Italy}
\EmailD{\href{mailto:dabrow@sissa.it}{dabrow@sissa.it}}

\Address{$^\ddag$~Institute of Physics, Jagiellonian University,\\
\hphantom{$^\ddag$}~Stanis\l{}awa \L{}ojasiewicza 11, 30-348 Krak\'ow, Poland}
\EmailD{\href{mailto:andrzej.sitarz@uj.edu.pl}{andrzej.sitarz@uj.edu.pl}}

\Address{$^\S$~Institute of Mathematics of the Polish Academy of Sciences,\\
\hphantom{$^\S$}~\'Sniadeckich 8, 00-656 Warszawa,  Poland}

\ArticleDates{Received December 09, 2014, in f\/inal form September 17, 2015; Published online September 26, 2015}

\Abstract{We introduce a family of spectral triples that describe the curved noncommutative two-torus. The relevant family of new Dirac operators is given by rescaling one of two terms in the f\/lat Dirac operator. We compute the dressed scalar curvature and show that the Gauss--Bonnet theorem holds (which is not covered by the general result of Connes and Moscovici).}

\Keywords{noncommutative geometry; Gauss--Bonnet; spectral triple}
\Classification{58B34; 46L87}

\section{Introduction}

In the seminal works \cite{cohcon, CoTr} Connes and Tretkof\/f initiated the investigation of curvature
aspects on the noncommutative two torus ${\mathbb T}_\theta^2$ and have shown the analogue of the Gauss--Bonnet theorem
for the conformally rescaled Dirac operator~$D$ and the related spin Laplacian corresponding to the standard
conformal structure. In~\cite{fatkha} these studies were extended to arbitrary conformal structure. The
scalar curvature itself was def\/ined and computed in~\cite{conmos2} and independently in~\cite{fatkha1}.

The methods used in these papers build on  Connes' pseudodif\/ferential calculus~\cite{CoTr}
and heat kernel short-time asymptotic expansion. The novelty therein is the employment of
{\em twisted} (or {\em non-unimodular}) spectral triples~\cite{cm2}, and as a consequence non-tracial weights and the modular operator. For some related papers see~\cite{bhumar12,fatkha2,fatwon}.

In the present paper we introduce a family of {\em usual} spectral triples in the sense of \cite{cm1},
known also as $K$-cycles~\cite[Def\/inition IV.2.11]{Co94}, that describe a curved~${\mathbb T}_\theta^2$.
The relevant new Dirac operators correspond to the rescaling one of two terms in the f\/lat Dirac operator by a positive element $k$ from the opposite coordinate algebra in the commutant.
This assures that the commutators of~$D$ with
the elements of the algebra are bounded.
Thus, to the best of our knowledge, this provides the f\/irst instance of usual spectral triples which describe non f\/lat (curved) noncommutative geometry.
It should also be stressed that partial rescalings by~$k$ from the algebra itself go outside the realm of spectral triples (even twisted), with geometric meaning (to the best of our knowledge) not yet fully clarif\/ied (in this respect see however, e.g.,~\cite{pw15}).

Unlike the classical case the new family of Dirac operators is not conformally related to the standard family of ``f\/lat'' (equivariant) Dirac operators.
Since they are not an overall rescaling of any of the so far known Dirac operators one cannot apply the general argument used by Connes and Moscovici~\cite[Theorem~2.1]{conmos2} to establish the Gauss--Bonnet theorem. Thus it is the f\/irst instance when one needs to ef\/fectively compute it, without certainty that it will be satisf\/ied.

This has been shown {\em perturbatively} up to second order in~\cite{DS}
for small modif\/ications of the f\/lat Dirac operator depending on four elements in the opposite algebra.
In the present paper we perform the {\em exact} computation, which shows that the Gauss--Bonnet theorem indeed extends to a much larger class of spectral triples.

Moreover, we compute the expression for the
dressed
scalar curvature as an element of the opposite algebra of~${\mathbb T}_\theta^2$.
In the classical case it corresponds to the product of the scalar curvature by the coef\/f\/icient of the metric volume form with respect to the Lebesgue volume.
It should be mentioned that such a dressing is actually tacitly present in all the previous computations.

The results obtained here are possible due to a recent neat generalisation by Lesch~\cite{lesch}
of the ``rearrangement lemma'', which is an important technical tool in~\cite{CoTr}.

The paper is organised as follows: in Section~\ref{section2} we brief\/ly present the
asymmetric rescaling of the classical f\/lat 2-torus and its curvature. Next we introduce a quantum (noncommutative) analogue of the associated Dirac operator, and the corresponding new family of spectral triples on the noncommutative torus.
In Section~\ref{section3}, which is the main part of the paper, we compute the value of the respective zeta function of the new Dirac operators at $z=0$ following the heat trace techniques and pseudodif\/ferential calculus elaborated in~\cite{conmos2, CoTr}, and an adaptation of the Lesch rearrangement lemma~\cite{lesch}.
In Section~\ref{section4} we comment on some algebra computations of the (undressed) curvature.

\section{Dirac operator}\label{section2}

Let  ${\mathbb T}^2$ be the classical torus with coordinates $(x, y)\in [0, 2\pi]^2$,
equipped with the metric
\begin{gather}
\label{metric}
dx^2 + k^{-2}(x,y) dy^2,
\end{gather}
where $k$ is a  positive function.
For example such an ``asymmetric torus'' metric with
\begin{gather*}
k^{-1} = c + \cos y
\end{gather*}
comes as the induced metric on the usual realisation of~${\mathbb T}^2$ as an embedded surface in~${\mathbb R}^3$
with the following parametrisation
\begin{gather*}
X	=	(c+\cos  y)  \cos  x, \qquad	
Y	=	(c+ \cos y)  \sin  x, \qquad	
Z	=	\sin y .
\end{gather*}
\begin{figure}[h!]\centering
\begin{minipage}[b]{50mm}\centering
\includegraphics[width=4cm]{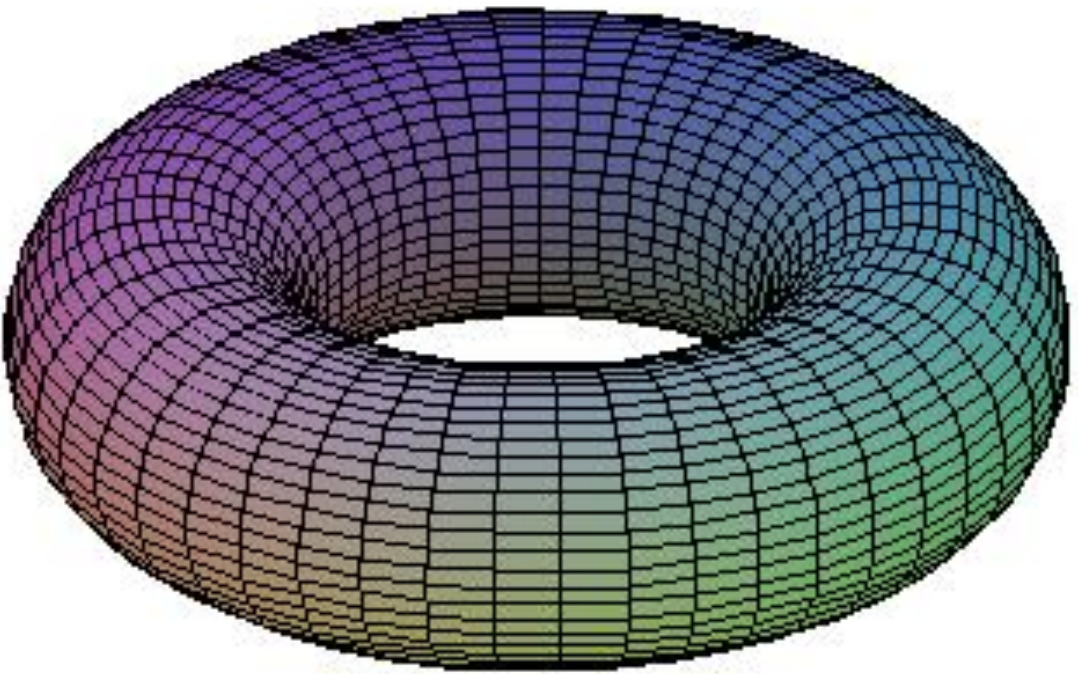}\\
Torus embedded in ${\mathbb R}^3$.
\end{minipage}\qquad
\begin{minipage}[b]{50mm}\centering
\includegraphics[width=4cm]{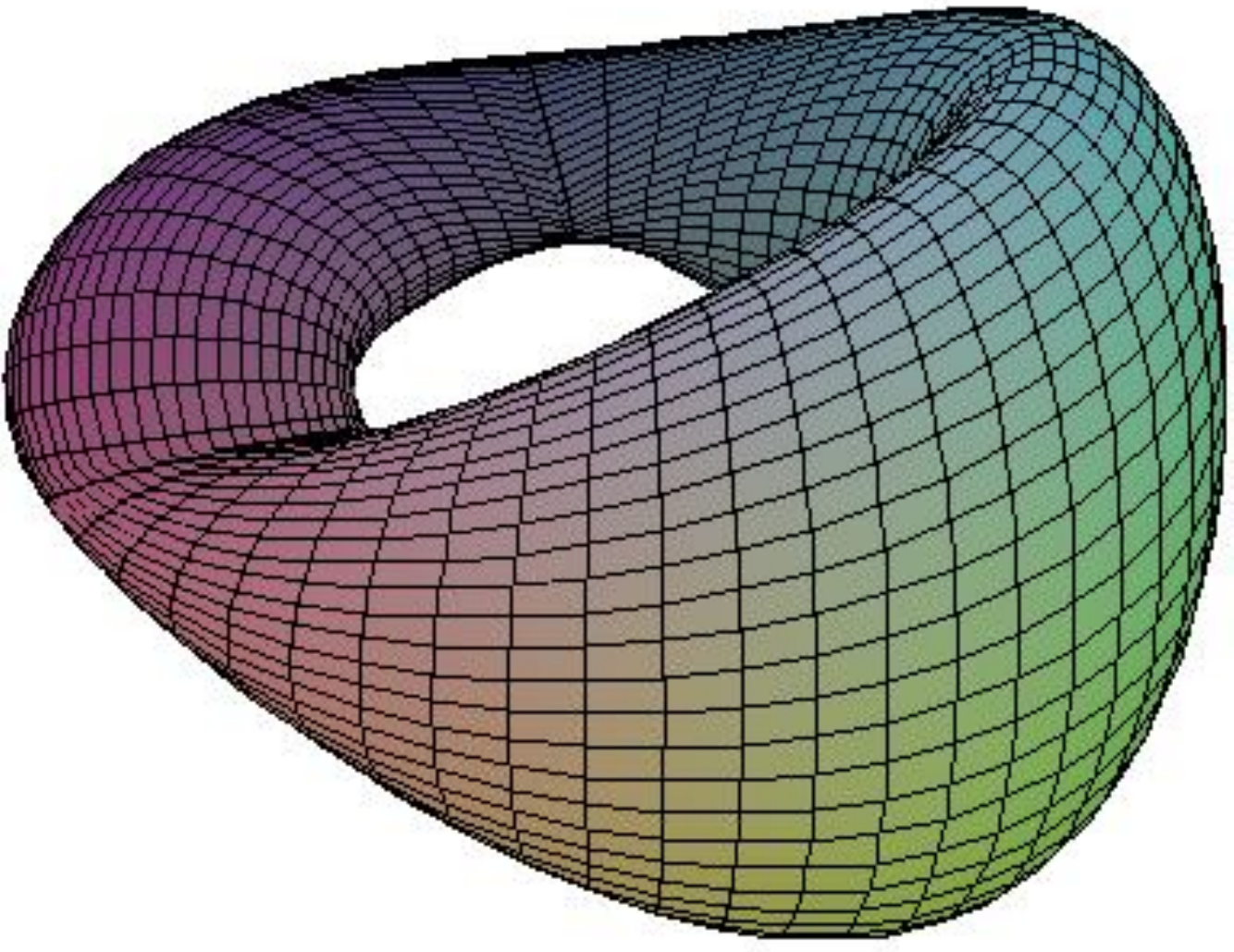} \\
 Asymmetric torus in ${\mathbb R}^3$.
\end{minipage}
\end{figure}

The scalar curvature of the torus with the metric \eqref{metric} reads
\begin{gather}
\label{classcurv}
R = 2k^{-1} \partial_x^2 (k) - 4 k^{-2} (\partial_x (k) )^2 .
\end{gather}

In the commutative case such metric is, of course, conformally equivalent to some f\/lat metric on the
torus even though the explicit formula for the curvature depends on the chosen coordinate system.
However, when passing to the noncommutative torus we are entering a new unexplored land, where
one does not know what is a {\em metric} and what is the exact meaning of {\em conformally equivalent}.
As in the approach of Connes the natural object is the Dirac operator rather than the metric itself.
For this reason we propose a new Dirac operator, which generalises to the noncommutative situation
the classical case of asymmetric torus.

We start with the commutative case of Dirac operator on $L^2({\mathbb T}^2, k^{-1} dx dy)\otimes {\mathbb C}^2$
for the metric~\eqref{metric}
\begin{gather*}
\tilde{D} = -i \sigma^1 \left(\partial_x -\tfrac{1}{2} k^{-1} \partial_x (k)  \right) -  i \sigma^2   k\  \partial_y ,
\end{gather*}
where
\begin{gather*}
\sigma^1= \left(
\begin{matrix}
0      & 1   \\
 1     &  0
\end{matrix}
\right),\qquad
\sigma^2= \left(
\begin{matrix}
0      & -i  \\
 i     &  0
\end{matrix}
\right).
\end{gather*}
Using the multiplication  by $\sqrt{k}$ we obtain a unitarily equivalent self-adjoint Dirac operator
\begin{gather}
\label{dirac}
\tilde{D} =  - i \sigma^1 \partial_x  -  i \sigma^2 \left(  k \partial_y + \tfrac{1}{2}  \partial_y(k)  \right)
\end{gather}
on the dense domain $H^{1,2}(T)$ in the Hilbert space $L^2(T,  dx dy)\otimes {\mathbb C}^2$.

Next we pass to the noncommutative torus ${\mathbb T}_\theta^2$ \cite{Co94}, about which we
collected the needed information in Appendix~\ref{appendixA}.
In particular we shall use therein
the usual trace $\mathfrak{t}$,
the usual self-adjoint derivations $\delta_1$, $\delta_2$,
and the standard real structure
(Tomita--Takesaki conjugation)~$J$ on the coordinate algebra~$A({\mathbb T}_\theta^2)$. Then we introduce the following new family of spectral triples, which is an immediate noncommutative analogue of~\eqref{dirac}.

\begin{dfn}
For $k \in J A({\mathbb T}_\theta^2) J^{-1}$ we set
\begin{gather*}
D_k =   \sigma^1 \delta_1 +   \sigma^2 \left(  k\, \delta_2 + \tfrac{1}{2}  \delta_2(k)  \right)
\end{gather*}
acting on ${\mathcal H}= L^2({\mathbb T}_\theta^2, \mathfrak{t})\otimes {\mathbb C}^2 $,
obtaining a family of spectral triples $({\mathbb T}_\theta^2, {\mathcal H},  D_k)$.
\end{dfn}

The form of the operator $D_k$ resembles the classical Dirac operator  $D$  \eqref{dirac}, however with the partial derivatives replaced by~$i\delta_1$,~$i\delta_2$.
Since we take~$k$ in the algebra
$J A({\mathbb T}_\theta^2) J^{-1}$ commuting with $A({\mathbb T}_\theta^2)$
due to the property~\eqref{eq:Jzero-LRcomm},
$D_k$ has bounded commutators with $a\in A({\mathbb T}_\theta^2)$.
Next the self-adjointness and the compact resolvent property follow from the fact that~$D_k$ is a perturbation of the f\/lat Dirac operator~$D$ (corresponding to $k=1$) by a (self-adjoint) bounded operator. Thus it def\/ines a spectral triple in the usual sense.

Furthermore, $D_k$ is a dif\/ferential operator in the sense of~\cite{CoTr} and we can extract the
associated (dressed) scalar curvature following the methods introduced and applied to
the global conformal rescaling of the noncommutative~\cite{conmos2,CoTr}. The idea
is to look for the functional:
\begin{gather*}
f \mapsto a_2(D_k; f),
\end{gather*}
where $f$ is an element of $J A({\mathbb T}_\theta^2) J^{-1}$ (or $A({\mathbb T}_\theta^2)$) and $a_2(D_k; f)$
is the second heat trace coef\/f\/icient of $\operatorname{Tr} (f e^{- t D_k^2})$. Since the computations
which use the pseudodif\/ferential operators allow us to express the results in terms of $\tau(f \widetilde{R})$
for an element~$\widetilde{R}$ of the algebra~$J A({\mathbb T}_\theta^2) J^{-1}$, one can interpret it as the
{\em dressed scalar curvature}, which in the classical case equals $\widetilde{R} = \sqrt{g} R$.

The noncommutative case is, in fact,  much richer, since the chiral heat trace functional
$\operatorname{Tr} (f\gamma e^{- t D_k^2})$ does not vanish identically. Using similar arguments
as above we can introduce the {\em dressed chiral scalar curvature}, $\widetilde{R_\gamma}$ as
an element of  $J A({\mathbb T}_\theta^2) J^{-1}$ such that the heat trace coef\/f\/icient~$a_2(D_k, f \gamma)$
is  expressed as~$\tau(f \widetilde{R}_\gamma)$.

\section{The curvature}\label{section3}
The square of $D$ reads
\begin{gather*}
(D_k)^2 = (\delta_1)^2 + k^2 (\delta_2)^2
           + \left( \tfrac{3}{2} k \delta_2(k) + \tfrac{1}{2} \delta_2(k) k + \sigma^1\sigma^2 \delta_1(k) \right) \delta_2 \\
\hphantom{(D_k)^2 =}{}
           + \tfrac{1}{4}  (\delta_2(k))^2 + \tfrac{1}{2} \sigma^1\sigma^2 \delta_{12}(k) + \tfrac{1}{2} k \delta_{22}(k),
\end{gather*}
and its symbol is
\begin{gather*}
\sigma((D_k)^2) = a_0 + a_1 +a_2,
\end{gather*}
where
\begin{gather*}
a_0 =  \xi_1^2 + k^2 \xi_2^2,  \qquad
a_1 = \left( \tfrac{3}{2} k \delta_2(k) + \tfrac{1}{2} \delta_2(k) k + \sigma^1\sigma^2 \delta_1(k) \right) \xi_2, \\
a_2 =  \tfrac{1}{4}  (\delta_2(k))^2 + \tfrac{1}{2} \sigma^1\sigma^2 \delta_{12}(k) + \tfrac{1}{2} k \delta_{22}(k).
\end{gather*}
As was demonstrated f\/irst in \cite{CoTr} the value $\zeta(0)$ at the origin of the zeta function
of the operator $(D_k)^2$ is given by
\begin{gather*}
\zeta(0)= - \int \mathfrak{t} (b_2(\xi))  d \xi,
\end{gather*}
where $b_2(\xi)$ is a symbol of order $-4$ of the pseudodif\/ferential operator $((D_k)^2+1)^{-1}$.
Here we assume that $(D_k)$ has no zero eigenvalues, which is satisf\/ied for generic $k$.
Otherwise, in the case that the Dirac operator has $0$ as an eigenvalue,
a constant term equal to $\dim (\operatorname{ker} (D_k))$ is present.
The term on the right hand side is interpreted as the quantum analogue of the integrated scalar curvature dressed by the square root of the determinant of the metric tensor.

The symbol $b_2(\xi)$ can be computed by pseudodif\/ferential calculus of symbols from the symbol
$a_2(\xi) + a_1(\xi) + a_0(\xi)$ of
$(D_k)^2$ as follows
\begin{gather}
b_0 =  (a_2+1)^{-1},\qquad
b_1 = -(b_0 a_1 b_0  + \partial_1(b_0) \delta_1(a_2) b_0 +
\partial_2(b_0)\delta_2(a_2) b_0), \nonumber\\
b_2=  -(b_0 a_0 b_0 + b_1 a_1 b_0 + \partial_1(b_0) \delta_1(a_1) b_0 +
\partial_2(b_0)\delta_2(a_1)b_0+ \partial_1(b_1)\delta_1(a_2)b_0
\nonumber \\
\hphantom{b_2=}{}  + \partial_2(b_1) \delta_2(a_2)b_0 + \tfrac{1}{2}\partial_{11}(b_0)\delta_1^2(a_2)b_0
  + \tfrac{1}{2}\partial_{22}(b_0)\delta_2^2(a_2)b_0  + \partial_{12}(b_0)\delta_{12}(a_2)b_0).
\label{b2}
\end{gather}

Since to obtain the curvature (or~$\zeta(0)$) we need to integrate
with respect to~$\xi_1$,~$\xi_2$, we notice that all the terms which contain odd powers of these variables shall vanish. Therefore, we can neglect them and keep for computations only the relevant parts with even powers. They read then as follows
\begin{gather*}
b_2^{\rm even} = A + B + C ,
\end{gather*}
where
\begin{gather*}
A=
 - 2 k b_0^2 \delta_1(k) k b_0 \delta_1(k) b_0 \xi_2^4
+ 4 k b_0^2 \delta_1(k) k b_0^2 \delta_1(k) b_0 \xi_1^2 \xi_2^4
- 2 k b_0^2 \delta_1(k) b_0 \delta_1(k) k b_0 \xi_2^4 \\
\hphantom{A=}{} + 4 k b_0^2 \delta_1(k) b_0^2 \delta_1(k) k b_0 \xi_1^2 \xi_2^4
+ 8 k b_0^3 \delta_1(k) k b_0 \delta_1(k) b_0 \xi_1^2 \xi_2^4
+ 8 k b_0^3 \delta_1(k) b_0 \delta_1(k) k b_0 \xi_1^2 \xi_2^4 \\
\hphantom{A=}{} - b_0 \delta_1(k) b_0 \delta_1(k) b_0 \xi_2^2
+ 2 b_0^2 \delta_1(k) \delta_1(k) b_0 \xi_2^2
- 2 b_0^2 \delta_1(k) k b_0 \delta_1(k) k b_0 \xi_2^4 \\
\hphantom{A=}{} + 4 b_0^2 \delta_1(k) k b_0^2 \delta_1(k) k b_0 \xi_1^2 \xi_2^4
- 2 b_0^2 \delta_1(k) k^2 b_0 \delta_1(k) b_0 \xi_2^4
+ 4 b_0^2 \delta_1(k) k^2 b_0^2 \delta_1(k) b_0 \xi_1^2 \xi_2^4 \\
\hphantom{A=}{} - 8 b_0^3 \delta_1(k) \delta_1(k) b_0 \xi_1^2 \xi_2^2
+ 8 b_0^3 \delta_1(k) k b_0 \delta_1(k) k b_0 \xi_1^2 \xi_2^4
+ 8 b_0^3 \delta_1(k) k^2 b_0 \delta_1(k) b_0 \xi_1^2 \xi_2^4,
\\
B=
\tfrac{15}{4} k b_0 \delta_2(k) k b_0 \delta_2(k) b_0 \xi_2^2
- 3 k b_0 \delta_2(k) k^2 b_0^2 \delta_2(k) k b_0 \xi_2^4
- 3 k b_0 \delta_2(k) k^3 b_0^2 \delta_2(k) b_0 \xi_2^4 \\
\hphantom{B=}{}+ \tfrac{9}{4} k b_0 \delta_2(k) b_0 \delta_2(k) k b_0 \xi_2^2
+ 6 k^2 b_0^2 \delta_2(k) \delta_2(k) b_0 \xi_2^2
- 8 k^2 b_0^2 \delta_2(k) k b_0 \delta_2(k) k b_0 \xi_2^4 \\
\hphantom{B=}{}- 10 k^2 b_0^2 \delta_2(k) k^2 b_0 \delta_2(k) b_0 \xi_2^4
+ 4 k^2 b_0^2 \delta_2(k) k^3 b_0^2 \delta_2(k) k b_0 \xi_2^6
+ 4 k^2 b_0^2 \delta_2(k) k^4 b_0^2 \delta_2(k) b_0 \xi_2^6 \\
\hphantom{B=}{}- 12 k^3 b_0^2 \delta_2(k) k b_0 \delta_2(k) b_0 \xi_2^4
+ 4 k^3 b_0^2 \delta_2(k) k^2 b_0^2 \delta_2(k) k b_0 \xi_2^6
+ 4 k^3 b_0^2 \delta_2(k) k^3 b_0^2 \delta_2(k) b_0 \xi_2^6 \\
\hphantom{B=}{}- 10 k^3 b_0^2 \delta_2(k) b_0 \delta_2(k) k b_0 \xi_2^4
- 8 k^4 b_0^3 \delta_2(k) \delta_2(k) b_0 \xi_2^4
+ 8 k^4 b_0^3 \delta_2(k) k b_0 \delta_2(k) k b_0 \xi_2^6 \\
\hphantom{B=}{}+ 8 k^4 b_0^3 \delta_2(k) k^2 b_0 \delta_2(k) b_0 \xi_2^6
+ 8 k^5 b_0^3 \delta_2(k) k b_0 \delta_2(k) b_0 \xi_2^6
+ 8 k^5 b_0^3 \delta_2(k) b_0 \delta_2(k) k b_0 \xi_2^6 \\
\hphantom{B=}{}- \tfrac{1}{4} b_0 \delta_2(k) \delta_2(k) b_0
+ \tfrac{3}{4} b_0 \delta_2(k) k b_0 \delta_2(k) k b_0 \xi_2^2
+ \tfrac{5}{4} b_0 \delta_2(k) k^2 b_0 \delta_2(k) b_0 \xi_2^2 \\
\hphantom{B=}{}- b_0 \delta_2(k) k^3 b_0^2 \delta_2(k) k b_0 \xi_2^4
- b_0 \delta_2(k) k^4 b_0^2 \delta_2(k) b_0 \xi_2^4,
\end{gather*}
and
\begin{gather*}
C =  k b_0^2 \delta_{11}(k) b_0 \xi_2^2
- 4 k b_0^3 \delta_{11}(k) b_0 \xi_1^2 \xi_2^2
  + b_0^2 \delta_{11}(k) k b_0 \xi_2^2
- 4 b_0^3 \delta_{11}(k) k b_0 \xi_1^2 \xi_2^2- \tfrac{1}{2} k b_0 \delta_{22}(k) b_0 \\
\hphantom{C=}{}
+ 2 k^2 b_0^2 \delta_{22}(k) k b_0 \xi_2^2
+ 4 k^3 b_0^2 \delta_{22}(k) b_0 \xi_2^2
 - 4 k^4 b_0^3 \delta_{22}(k) k b_0 \xi_2^4
- 4 k^5 b_0^3 \delta_{22}(k) b_0 \xi_2^4.
\end{gather*}

Similarly for the chiral part
$\widetilde{R}_\gamma$ of the curvature, we have
\begin{gather*}
b_{\gamma 2}^{\rm even} =
A_{\gamma } + B_{\gamma } + C_{\gamma },
\end{gather*}
where
\begin{gather*}
A_{\gamma } =
- 2 k^2 b_0^2 \delta_1(k) k b_0 \delta_2(k) b_0 i \xi_2^4
- 2 k^2 b_0^2 \delta_1(k) b_0 \delta_2(k) k b_0 i \xi_2^4 + \tfrac{5}{2} b_0 \delta_1(k) k  b_0 \delta_2(k) b_0 i \xi_2^2\\
\hphantom{A_{\gamma } = }{}
- 2 b_0 \delta_1(k) k^2 b_0^2 \delta_2(k) k b_0 i \xi_2^4
 - 2 b_0 \delta_1(k) k^3 b_0^2 \delta_2(k) b_0 i \xi_2^4
+ \tfrac{3}{2} b_0 \delta_1(k) b_0 \delta_2(k) k b_0 i \xi_2^2,
\\
 B_{\gamma } =
   \tfrac{3}{2} k b_0 \delta_2(k) b_0 \delta_1(k) b_0 i \xi_2^2
- 2 k^2 b_0^2 \delta_2(k) k b_0 \delta_1(k) b_0 i \xi_2^4 \\
\hphantom{B_{\gamma } = }{}
 - 2 k^3 b_0^2 \delta_2(k) b_0 \delta_1(k) b_0 i \xi_2^4
+ \tfrac{1}{2} b_0 \delta_2(k) k b_0 \delta_1(k) b_0 i \xi_2^2,
\end{gather*}
and
\begin{gather*}
C_{\gamma } =
2 k^2 b_0^2 \delta_{12}(k) b_0 i \xi_2^2
- \tfrac{1}{2} b_0 \delta_{12}(k) b_0 i.
\end{gather*}

\subsection{The classical limit}
At this point we can check the classical (commutative) value of our expressions for $\theta =0$.
The symbol $b_2$ becomes
\begin{gather*}
b_2(\xi) =     48 b_0^5 k^6 \delta_2(k)^2 \xi_2^6 + 48 b_0^5 k^2 \delta_1(k)^2 \xi_1^2 \xi_2^4 - 8 b_0^4 k^5 \delta_{22}(k) \xi_2^4
           - 56 b_0^4 k^4 \delta_2(k)^2 \xi_2^4 \\
\hphantom{b_2(\xi) = }{}
           - 8 b_0^4 k^2 \delta_1(k)^2 \xi_2^4 - 8 b_0^4 k \delta_{11}(k) \xi_1^2 \xi_2^2
           - 8 b_0^4 \delta_1(k)^2 \xi_1^2 \xi_2^2 + 6 b_0^3 k^3 \delta_{22}(k) \xi_2^2 \\
 \hphantom{b_2(\xi) = }{}
           + 14  b_0^3 k^2 \delta_2(k)^2 \xi_2^2
 + 2 b_0^3 k \delta_{11}(k) \xi_2^2 + b_0^3 \delta_1(k)^2 \xi_2^2 - \tfrac{1}{2} b_0^2 k \delta_{22}(k) - \tfrac{1}{4} b_0^2 \delta_2(k)^2,
\end{gather*}
which after integration gives
\begin{gather*}
\int b_2 d\xi_1 d\xi_2
 = - \frac{\pi}{3} \frac{(\delta_1(k))^2}{k^3} + \frac{\pi}{6}  \frac{\delta_{11}(k)}{k^2}.
\end{gather*}
Taking into account that we compute the Gilkey--Seeley--deWitt coef\/f\/icients of the
heat kernel asymptotic for
the square of the Dirac operator and not the Laplace operator itself,
and assuming as above that $D_k$ has no zero eigenvalues, we have
\begin{gather*}
\zeta(0) =  \frac{1}{48 \pi} \int \sqrt{g} R.
\end{gather*}

The assumption that $0$ is not in the spectrum of $D_k$, is satisf\/ied for generic $k$.
Otherwise, in the case that~$D_k$  has $0$ as an eigenvalue, there is
a constant term contribution equal to~$\dim (\operatorname{ker} D_k)$.

Moreover, since
$\mathfrak{t} =  \frac{1}{4\pi^2} \int dx dy$ for $\theta = 0$, taking into account
the appropriate rescaling of the volume form and putting it all together we obtain
\begin{gather*}
\sqrt{g} R = 48 \pi \frac{1}{4\pi^2}
\left( - \frac{\pi}{3} \frac{(\partial_1(k))^2}{k^3} + \frac{\pi}{6}  \frac{\partial_{11}(k)}{k^2} \right)
= \left( 2 k^{-2} \partial_{11}(k) - 4 k^{-3} (\partial_1(k))^2 \right),
\end{gather*}
which agrees with the classical formula~(\ref{classcurv}).
Similarly, the symbol $ b_{\gamma 2}$ becomes
\begin{gather*}
b_{\gamma 2}(\xi) =
       - 12 b_0^4 k^3 \delta_1(k) \delta_2(k) i \xi_2^4 + 2 b_0^3 k^2 \delta_{12}(k) i \xi_2^2 + 6 b_0^3 k
      \delta_1(k) \delta_2(k) i \xi_2^2 - \tfrac{1}{2} b_0^2 \delta_{12}(k) i,
\end{gather*}
which after integration yields
\begin{gather*}
\int d\xi_1 d\xi_2   b_{\gamma 2} = 0
\end{gather*}
and thus
\begin{gather*}  \sqrt{g} R_\gamma = 0.
\end{gather*}

Before we can proceed with the noncommutative computation let us recall the general
framework of computations as shown recently by Lesch~\cite{lesch}.

\subsection{Rearrangement lemma}

In \cite{lesch} Lesch  proved the following formula
(with a slightly modif\/ied notation)
\begin{gather*}
\int_0^\infty  f_0\big(uk^2\big)   a_1
f_1\big(uk^2\big)  a_2 \cdots a_p   f_p\big(uk^2\big) du \nonumber \\
\qquad {} = k^{-2} m\, F\big(\md{1},\md{1}\md{2},\ldots, \md{1} \cdots \md{p}\big)(a_1 \otimes a_2 \otimes \cdots\otimes a_p),
\end{gather*}
where $m$ is the $p$-fold multiplication, the function $F$ of $p$ variables reads
\begin{gather*}
F(s_1,\ldots,s_p) = \int_0^\infty f_0(u)f_1(us_1)\cdots f_p(u s_p) du
\end{gather*}
and $\md{j}$ signif\/ies the square of the modular operator~$\Delta$, $\Delta (a) = k^{-1} a k$,
acting on the $j$-th factor $a_j$ in the $p$-fold
tensor product product $a_1 \otimes a_2 \otimes \cdots\otimes a_p$.

In our case we need to adapt the formula to a slightly dif\/ferent setting, when we integrate over two
variables~$\xi_1$ and~$\xi_2$. All the integrals we have are of the form
\begin{gather*}
{\mathcal J} = \int_{- \infty}^\infty d\xi_1 \int_{- \infty}^\infty d\xi_2
k^{n_1} b_0^{m_1}(\xi_1, \xi_2)   X   k^{n_2} b_0^{m_2}(\xi_1, \xi_2)   Y   k^{n_3} b_0^{m_3}(\xi_1, \xi_2) \xi_1^{2 k_1} \xi_2^{2 k_2}, \end{gather*}
where
$X$, $Y$ are some derivatives of $k$ and
\begin{gather*}
b_0(\xi_1, \xi_2) = \frac{1}{1 + \xi_1^2 + k^2 \xi^2}.
\end{gather*}

Extending the result of Lesch we see that
\begin{gather*}
{\mathcal J} = F\big(\mo{1},\mo{1} \mo{2}\big)(X \cdot Y),
\end{gather*}
where after change of variables we obtain
\begin{gather*}
F(s,t) =  2 \int_{0}^\infty \!d v \int_{0}^\infty\! d u     k^{n_1+ n_2+ n3 -1 -2k_2}  \frac{u^{k_2 - \frac{1}{2}} v^{2 k_1} }{(1\!+\! v^2\! +\! u)^{m_1}}
 \frac{s^{n_2}}{(1\!+\! v^2\! + \!u s^2 )^{m_2}}  \frac{t^{n_3}}{(1\!+\! v^2\! + \!u t^2 )^{m_3}} .
 \end{gather*}

In case $Y=1$ the resulting function depends only on $s$.
\subsection{The curvature and its trace}

In order to compute explicitly the expressions for the curvature we shall use the following lemma.
\begin{lem}
Under the trace an entire function $F$ of two variables satisfies
\begin{gather*}
\mathfrak{t} \big(  F\big(\mo{1},\mo{1} \mo{2}\big)(X \cdot Y) \big) =
\mathfrak{t} \big(  F\big(\mo{1}, {\rm id} \big)(X Y)\big) = \mathfrak{t} \big(F\big(\mo{1},1\big)(X) Y\big).
\end{gather*}
and in case of one variable
\begin{gather*}
\mathfrak{t} \big(  F\big(\mo{1}\big)(X) \big) = \mathfrak{t}  (  F(1) X  ).
\end{gather*}
\end{lem}

\begin{proof}
We have
\begin{gather*}F(s,t) = \sum_{n,m \geq 0} f_{nm} s^n t^m,
\end{gather*}
so
\begin{gather*}
\mathfrak{t} \big(  F\big(\mo{1},\mo{1} \mo{2}\big)(X \cdot Y) \big)   = \sum_{n,m \geq 0} f_{nm} \mathfrak{t} \big(
\Delta^{n+m}(X) \Delta^m (Y) \big) \\
\qquad = \sum_{n,m \geq 0} f_{nm} \mathfrak{t} \big( \Delta^m \big( \Delta^n(X) Y \big) \big)
  = \sum_{n,m \geq 0} f_{nm} \mathfrak{t} \big(\Delta^n(X) Y  \big)
  = \mathfrak{t} \big(  F\big(\mo{1}, 1\big)(X \cdot Y) \big).
\end{gather*}
The other identity is a simple consequence of the above one.
\end{proof}

\subsection{Curvature and chiral curvature}

We can present the f\/irst main result on the dressed scalar curvature of the asymmetric noncommutative torus:

\begin{pro}\label{main}
The dressed scalar curvature for the asymmetric torus is
\begin{gather*}
\widetilde{R} = F_{11}\big(\mo{1},\mo{1} \mo{2}\big)(\delta_1(k) \cdot\delta_1(k)) +  F_{22}\big(\mo{1},\mo{1} \mo{2}\big)(\delta_2(k), \delta_2(k)) \\
\hphantom{\widetilde{R} =}{}
+
F_{11}'\big(\mo{1}\big)(\delta_1(k)^2)  + F_{22}'\big(\mo{1}\big)(\delta_2(k)^2)
+ F_{1}\big(\mo{1}\big)(\delta_{11}(k)) + F_2\big(\mo{1}\big)(\delta_{22}(k)),
\end{gather*}
where
\begin{gather*} F_{11}(s,t) =  - \frac{2\pi }{3k^3} \frac{(2 s^2 + 4st + 4s +3 + 8t + 3t^2)}{(t+1)^3 (s+1) (s+t) },\qquad
 F_{22}(s,t) = \frac{\pi }{2k} \frac{(t^2 - 6t +1)}{(t+1)^3 },\\
 F_{11}'(s) =   \frac{4\pi }{3k^3} \frac{1}{(s+1)^3 }, \qquad
 F_{22}'(s) = - \frac{\pi }{2k} \frac{(s^2 - 6s + 1)}{(s+1)^3 },
\end{gather*}
and
\begin{gather*} F_1(s) =  \frac{2\pi }{3k^2} \frac{1}{(s+1)^2 }, \qquad
  F_2(s) = 0. \end{gather*}
The trace of $\widetilde{R}$ vanishes.
\end{pro}

\begin{proof}
First of all, observe that
\begin{gather*}
F_{22}(s,1) + F_{22}'(1) = 0, \qquad F_2(1) = 0,
\end{gather*}
so all terms containing $\delta_2(k)$ and $ \delta_{22}(k)$ vanish.

For the terms containing $\delta_1(k)$ we have
\begin{gather*}
F_{11}(s,1) + F_{11}'(1) =  - \frac{\pi}{3 k^3} \frac{s+3}{(s+1)^2 }.
\end{gather*}
Then using the identity
\begin{gather*}
\mathfrak{t}\left( k^{-2} \delta_{11}(k) \right) = 2 \mathfrak{t}\left(  k^{-2} \delta_1(k) k^{-1} \delta_1(k) \right) =
2 \mathfrak{t}\left(  k^{-3} \Delta^{-1}(\delta_1(k)) \delta_1(k) \right),
\end{gather*}
which follows directly from the Leibniz rule and the fact that the trace is closed, we can rewrite the trace of the sum of the remaining  terms $F_{11}$, $F_{11}'$ and $F_{1}$ as
\begin{gather*}
 \mathfrak{t} \left( k^{-3} H(\Delta) (\delta_1(k) ) \delta_1(k) \right),
\end{gather*}
where
\begin{gather*}
H(s) =  \frac{\pi}{3 k^3} \frac{1-s}{s (s+1)^2 }.
\end{gather*}

Next, we observe that for any $A$ and $B$ and an entire function $H$
\begin{gather*}
\mathfrak{t} \left( k^{-3} H(\Delta)(A) B \right) = \mathfrak{t} \left(  H(\Delta)\big(\Delta^3(A) \big) k^{-3} B \right) =  \mathfrak{t} \left(  k^{-3} B H(\Delta)\big(\Delta^3(A) \big) \right),
\end{gather*}
and
\begin{gather*}
\mathfrak{t} \left( k^{-3} H(\Delta)(A) B \right) = \mathfrak{t} \left(  k^{-3} A H\big(\Delta^{-1}\big)(B) \right).
\end{gather*}

Now if $A=B$ then both expressions on the right-hand side are identical. In our case, however
\begin{gather*}
H(s) s^3 = \frac{\pi}{3 k^3} \frac{s^2 (1-s)}{(s+1)^2 }
\qquad \text{and}\qquad
H\big(s^{-1}\big) = \frac{\pi}{3 k^3} \frac{s^2 (s-1)}{(s+1)^2 },
\end{gather*}
and therefore since
\begin{gather*}
H(s) s^3 = - H\big(s^{-1}\big),
\end{gather*}
the trace of the above expression must vanish, hence, the Gauss--Bonnet theorem holds.
\end{proof}

Next we give the second main result on the dressed chiral curvature of the asymmetric noncommutative torus:

\begin{pro}\label{main2}
The dressed chiral curvature for the asymmetric torus is
\begin{gather*} \widetilde{R}_\gamma = G_{12}\big(\mo{1},\mo{1} \mo{2}\big)(\delta_1(k), \delta_2(k)) \\
\hphantom{\widetilde{R}_\gamma =}{}
+ G_{21}\big(\mo{1},\mo{1} \mo{2}\big)(\delta_2(k), \delta_1(k)) + G\big(\mo{1}\big)(\delta_{12}(k)),
\end{gather*}
where
\begin{gather*} G_{12}(s,t) = \frac{\pi }{k^2} \frac{(t-1)}{(t+1)^2 (s+1)}, \qquad
 G_{21}(s,t) = \frac{\pi }{k^2} \frac{(t-1)}{(t+1)^2 (s+t)},
 \qquad
G(s) = - \frac{\pi }{k} \frac{(s-1)}{(s+1)^2 }.
\end{gather*}
The trace of~$R_\gamma$ vanishes.
\end{pro}

\begin{proof}
By computation.
Then the last statement follows from
\begin{gather*}
G_{12}(s,1) = G_{21}(s,1) = G(1) = 0. \tag*{\qed}
\end{gather*}
\renewcommand{\qed}{}
\end{proof}

\section{Final comments}\label{section4}

\looseness=-1
We have introduced a new class of spectral triples that describe a curved geometry of the noncommutative torus.
The associated Dirac operators are obtained
through a partial rescaling of the f\/lat Dirac operator by elements from the opposite coordinate algebra.
It turns out that the dressed scalar curvature does not vanish,
but the Gauss--Bonnet theorem holds for that family.

\looseness=-1
Let us stress that even though in the classical limit they arise from the metric, which is conformally equivalent to the f\/lat one, this is an open problem in the noncommutative situation. An interesting question is if and in which sense this
can be established also in the noncommutative case.

It becomes now more evident that the class of admissible Dirac operators on the noncommutative torus is certainly bigger than the one-parameter family of ``f\/lat metric'' (equivariant) Dirac operators.
It is therefore necessary to study the conditions and the general setup of such construction, and we
conjecture that the class should contain all operators proposed in~\cite{DS} and their f\/luctuations.

Although in this paper we have concentrated on the 2-dimensional case, it is natural to consider
generalisations of the result to $3$ and more dimensions. In particular it is interesting
to study the curvature and minimality of yet another class Dirac operators introduced in
\cite{DS2}, which are similar to the ones encountered here.

Finally, we can comment on the purely algebraic computation by  Rosenberg~\cite{RsNC}
of the curvature and Gauss--Bonnet term for the conformally rescaled metric on the noncommutative torus using the noncommutative notion of Levi-Civita connection.
His method can be extended to the case we consider and a direct computation yields the Gauss--Bonnet functional
\begin{gather*}
2 {\mathfrak t} \left( - \tfrac{9}{4} k^{-2} \delta_1(k) k^{-1} \delta_1(k)
+ \tfrac{1}{4}  k^{-3} \delta_1(k) \delta_1(k) + k^{-2} \delta_{11}(k) \right),
\end{gather*}
which however does not vanish,
and the Gauss--Bonnet theorem does not hold.

Instead there is an even simpler possibility that yields an algebraic expression for a candidate for the scalar curvature.
Namely, we note that the formula~(2.1) in~\cite{DS} for the classical scalar curvature in terms of the orthonormal basis of vector f\/ields ({\em moving frame} or {\em two-bein}) for the partially rescaled metric, makes a perfect sense also for the noncommutative torus
(there is no ordering ambiguity).
Taking the two-bein as $(\delta_1, k \delta_2)$ one obtains
\begin{gather*} {\mathfrak t} \left( - 4  k^{-2} \delta_1(k) k^{-1} \delta_1(k)
+  2k^{-2} \delta_{11}(k) \right),\end{gather*}
which vanishes due to the properties of the trace.

\appendix

\section{Appendix}\label{appendixA}

{\sloppy The unital $*$-algebra ${\mathcal A}_\theta$ of smooth functions on the noncommutative 2-torus
consists of power series of the form
\begin{gather*} a =\sum_{m,n \in{\mathbb Z}^2} a_{mn}  U_1^m U_2^n,
 \end{gather*}
where the (double) sequence of coef\/f\/icients $\{a_{mn} \in {\mathbb C}, \, (m,n) \in{\mathbb Z}^2 \}$
decreases rapidly at `inf\/inity'
\begin{gather*}
\| a\|_k := \sup_{m,n} \big\{\big(1+m^2+n^2\big)^k  |a_{mn}|^2\big\} < \infty ,\qquad\forall\, k\in {\mathbb N} .
\end{gather*}
The $*$-algebra structure is specif\/ied by the commutation relations of the unitary genera\-tors~$U_1$,~$U_2${\samepage
\begin{gather*}
U_1  U_2 = e^{2\pi i \theta} U_2  U_1,
\end{gather*}
where $\theta$ is a real parameter.}

}

On $A_\theta$ there is a
(unique if  $\theta$ is irrational)
normalised faithful positive def\/inite tracial state
$\mathfrak{t} \colon {\mathcal A}_\theta \rightarrow {\mathbb C}$; it is given by
\begin{gather*}
\mathfrak{t}(a) =
\mathfrak{t} \bigg(\sum_{(m,n) \in {\mathbb Z}^2} a_{mn} ~U_1^m U_2^n\bigg) := a_{00}.
\end{gather*}
We call $\mathfrak{t}$ {\em trace}, for brevity.
It is invariant under the action of the commutative torus ${\mathbb T}^2= U(1)\times U(1)$
on ${\mathcal A}_\theta$ by automorphisms,
$ U_1\mapsto z_1U_1$ and $ U_2\mapsto z_2U_2$, where $ |z_1|=1= |z_2|$.
The inf\/initesimal action is determined by two commuting derivations $\delta_1$, $\delta_2$
acting by
\begin{gather*}
\delta_\mu (U_\nu) = 2\pi i  \delta_\mu^\nu U_\nu, \qquad \mu,\nu = 1, 2.
\end{gather*}
The invariance means just that $\mathfrak{t}(\delta_\mu a) = 0$, for $\mu = 1, 2$
and any $a\in{\mathcal A}_\theta$.

There is a sesquilinear form on ${\mathcal A}_\theta$
\begin{gather*}
(a,b)= \mathfrak{t}(a^*b)
\end{gather*}
and a norm
\begin{gather*}
\|a\| := \sqrt{\mathfrak{t}(a^*a)}.
\end{gather*}
The completion of ${\mathcal A}_\theta$ in this norm yields a Hilbert space
$H_{\rm o}:= L^2({\mathcal A}_\theta ,\mathfrak{t} )$,
which carries a~$*$-representation of ${\mathcal A}_\theta$ by left multiplication operators
\begin{gather*}
\pi(a)\colon \  b \mapsto ab.
\end{gather*}
Moreover $\pi$ is the GNS representation w.r.t.\ the state~$\mathfrak{t}$,
 the cyclic and separating vector is $1\in{\mathcal A}_\theta\subset H_{\rm o}$.
The operator norm of~$\pi(a)$ is~$\|a\|$.
The associated {\em Tomita conjugation} provides an antiunitary operator $J_{\rm o}$ on~$H_{\rm o}$,
\begin{gather*}
J_{\rm o}({b}) := {b^*}.
\end{gather*}
Clearly $J_{\rm o}^2 = 1$ and
\begin{gather*}
\pi^{\rm op}(a) := J_{\rm o} \pi(a^*) J_{\rm o} \colon \  b \mapsto ba
\end{gather*}
(right multiplication by~$a$) is a representation of the opposite algebra of~${\mathcal A}_\theta$.
Since the left and right multiplications commute, we have
\begin{gather}
[\pi(a), J_{\rm o} \pi(b^*) J_{\rm o}] = 0, \qquad  \forall\, a,b.
\label{eq:Jzero-LRcomm}
\end{gather}
(In the paper we are omitting the symbol~$\pi$.)

For Dirac spinors (Hilbert space)
we take the trivial module over $A_\theta$ of rank two
\begin{gather*}
H=H_{\rm o}\otimes {\mathbb C}^2 \simeq H_{\rm o} \oplus H_{\rm o} =: H_+ \oplus H_-  .
\end{gather*}
The inner product is $\langle (a,a^\prime),(b,b^\prime)\rangle=\mathfrak{t}(a^*b+a^\prime b^{\prime*})$.
The \textit{charge conjugation} operator on~$H$ is given by
\begin{gather*}
J := -iJ_{\rm o} \otimes (\sigma_2\circ {\rm c.c.}).
\end{gather*}

\subsection*{Acknowledgements}

L.D.~gratefully acknowledges the hospitality of the Institute of Physics, Jagiellonian University in Krak\'ow.
L.D.~partially supported by PRIN 2010 grant ``Operator Algebras,
Noncommutative Geometry and Applications'', A.S.~partially supported by NCN grant 2012/06/M/ST1/00169.
 The authors express their gratitude to the referees for valuable comments.

\pdfbookmark[1]{References}{ref}
\LastPageEnding

\end{document}